\documentclass[12pt,a4paper]{amsart}
\usepackage[english]{babel}
\usepackage{latexsym}
\usepackage{amssymb}
\usepackage{amscd}

\usepackage[all]{xy}
\usepackage[cp850]{inputenc}
\usepackage[mathscr]{eucal}
\newcommand{\OP}{\mathfrak}

\theoremstyle{plain}

\newtheorem*{prob*}{Problem}
\newtheorem{theorem}{Theorem}[section]
\newtheorem{proposition}{Proposition}[section]
\newtheorem{cor}{Corollary}[section]
\newtheorem{lemma}{Lemma}[section]
\theoremstyle{remark}

\newcommand{\R}{\mathbb{R}}
\newcommand{\N}{\mathbb{N}}
\newcommand{\K}{\mathbb{K}}

\newcommand{\ELP}{{\sf E}\hspace{-1pt}{\sf L}\hspace{-1pt}{\sf P }}

\newcommand\dist{\mathop{\mathrm{dist}}\nolimits}
\newcommand{\KP}{{\sf K}\hspace{-1pt}{\sf P}}

\newcommand{\adef}{\begin{defn}}
\newcommand{\zdef}{\end{defn}}
\newtheorem{defn}[theorem]{Definition}

\newcommand{\To}{\longrightarrow}

\def\Ext{\operatorname{Ext}}

\def\supp{\operatorname{supp}}
\def\PB{\operatorname{PB}}
\def\PO{\operatorname{PO}}

\def\Ext{\operatorname{Ext}}

\DeclareMathOperator{\bilipemb}{LipEmb}

\newcommand{\Z}{\sf Z}

\newcommand{\aproof}{\begin{proof}}
\newcommand{\zproof}{\end{proof}}

\makeatletter
\def\@tocline#1#2#3#4#5#6#7{\relax
    \ifnum #1>\c@tocdepth 
    \else
    \par \addpenalty\@secpenalty\addvspace{#2}%
    \begingroup \hyphenpenalty\@M
    \@ifempty{#4}{%
        \@tempdima\csname r@tocindent\number#1\endcsname\relax
    }{%
        \@tempdima#4\relax
    }%
    \parindent\z@ \leftskip#3\relax \advance\leftskip\@tempdima\relax
    \rightskip\@pnumwidth plus4em \parfillskip-\@pnumwidth
    #5\leavevmode\hskip-\@tempdima
    \ifcase #1
    \or\or \hskip 1em \or \hskip 2em \else \hskip 3em \fi%
    #6\nobreak\relax
    \hfill\hbox to\@pnumwidth{\@tocpagenum{#7}}\par
    \nobreak
    \endgroup
    \fi}
\makeatother

\title{A twisted Hilbert space not isomorphic to its dual}

\author{J.M.F. Castillo}
\address{Universidad de Extremadura, Instituto de Matem\'aticas Imuex, Avenida de Elvas, 06011-Badajoz, Spain} \email{castillo@unex.es}

\author{W.H.G. Corr\^{e}a }
\address{Departamento de Matem\'atica, Instituto de Ciencias Matem\'ticas e de Computa\c{a}o, Universidade de S\u{a}o Paulo, Avenida Trabalhador S\u{a}o-carlense, 400 - Centro, CEP: 13566-590, S\u{a}o Carlos, SP, Brazil.}
\email{\texttt{willhans@icmc.usp.br}}

\thanks{2020 AMS Class. number. 46M18, 46B20, 46B70, 46C99}


\thanks{The research of the second (maybe) author has been supported by Project PID2023-146505NB-C21 funded by MCIU/AEI/10.13039/5011000110331/FEDER/UE and Project IB24002 funded by la Junta de Extremadura. Willian Corr\^{e}a was supported by S\^{a}o Paulo Research Foundation (FAPESP), grants 2023/06973-2 and 2023/12916-1, and National Council for Scientific and Technological Development - CNPq - Brasil, grant 304990/2023-0.}

\begin{document}

\maketitle
\begin{abstract} We show: 1) The existence of the first twisted Hilbert space that is not isomorphic to its dual; this solves a problem posed by Cabello in  [\emph{Nonlinear centralizers in homology}, Math. Ann. 358 (2014), no. 3-4, 779--798]. 2) The existence of a large coneable family of relatively incomparable such examples, improving the coneable family obtained in [W.H. Corr\^{e}a, S. Dantas, D.L. Rodr\'iguez-Vidanes, \emph{Twisted Hilbert spaces defined by Lipschitz embeddings}, Israel J. of Mathematics, to appear]. 3) The existence of quasilinear maps between Hilbert spaces not isomorphic to Kalton centralizers; which solves another question of Cabello. 4) The existence of a large family of mutually incomparable elements in the ordered set of twisted Hilbert exact sequences. This complements earlier results in
[J.M.F. Castillo, W. Cuellar, V. Ferenczi, Y. Moreno, \emph{Complex structures on twisted Hilbert spaces},  Israel J. Math. 222 (2017) 787--814] --where it was proved that the ordered set did not have a first element-- and [F. Cabello S\'anchez, J.M.F. Castillo, W.H.G. Corr\^{e}a, V. Ferenczi, R. Garc\'ia, \emph{On the $\Ext^2$-problem in Hilbert spaces}, J. Funct. Anal. 280 (2021) 108863] --where two incomparable elements were obtained.\end{abstract}

\section{Introduction}

The reader is invited to go to the Preliminaries section for all unexplained terms appearing in this introduction.
In this paper we answer negatively the problem posed by Cabello \cite{cabeann}: Is every twisted Hilbert space isomorphic to its dual? Relevant information about this question was provided by Kalton \cite{kcomplex} with the construction of complex twisted Hilbert spaces $Z_2(\alpha)$ real isomorphic but not complex isomorphic to their duals. The exact sequence associated to each $Z_2(\alpha)$ was shown in \cite[Proposition 6.10]{cfg} to be \emph{singular} and several \emph{ad hoc} manipulations in \cite[Theorem 6.1]{cd} generate an entire cone of singular twisted Hilbert spaces real isomorphic but not complex isomorphic to their duals.

Up to now, most of the known twisted Hilbert sequences were obtained using \emph{centralizers} which, by standard results \cite{kaltdiff,cabeann,ccc}, must be isomorphic to their dual sequences (and, consequently, the twisted Hilbert space must be isomorphic to its dual). This suggests another problem posed, to the best of our knowledge, by Cabello: Is every twisted Hilbert sequence isomorphic to a twisted Hilbert sequence generated by a centralizer? We will show the existence of twisted Hilbert spaces isomorphic to their duals but such that their associated twisted Hilbert sequence is not isomorphic to the dual sequence and, therefore, cannot be isomorphic to a sequence generated by a centralizer. Using those constructions we will obtain a large \emph{coneable} family of relatively incomparable singular twisted Hilbert spaces not isomorphic to their duals. Recall that in \cite{cd} the first coneable family of mutually incomparable singular twisted Hilbert spaces was obtained. We apply those ideas to the study of the ordered set $\Ext(\ell_2, \ell_2)$  of twisted Hilbert exact sequences, showing that the coneable family above is formed by mutually incomparable elements. This improves two previous results in the area: That the ordered set $\Ext(\ell_2, \ell_2)$  has no minimum \cite{ccfmcomplex} and that the order is not total \cite{ext2}.\smallskip

Since the problems are related to the homological structure of twisted Hilbert spaces, let us begin with a short review of the topic, explaining the terminology and the relevance of the questions. The reader is invited to consult \cite{hmbst} for a sound background on exact sequences, quasilinear maps and the examples and homological constructions we will use.

\section{Preliminaries}
\subsection{Exact sequences of Banach spaces}

An exact sequence of Banach spaces is a diagram $\xymatrixcolsep{0.5cm}\xymatrix{0\ar[r]& Y  \ar[r]& Z \ar[r]& X \ar[r]&0}$
formed by Banach spaces and linear bounded operators such that the image of each operator coincides with the kernel of the next one. Two exact sequences $\xymatrixcolsep{0.5cm}\xymatrix{0\ar[r]& Y  \ar[r]& Z \ar[r]& X \ar[r]&0}$
and $\xymatrixcolsep{0.5cm}\xymatrix{0\ar[r]& Y  \ar[r]& Z' \ar[r]& X \ar[r]&0}$ are called equivalent when there exists a linear bounded operator $\beta:Z\To Z'$ making a commutative diagram
$$\xymatrix{
0\ar[r]&Y\ar@{=}[d]\ar[r]& Z\ar[d]^\beta\ar[r]& X\ar[r]\ar@{=}[d] & 0\\
0\ar[r]&Y'\ar[r]& Z'\ar[r]& X'\ar[r]& 0}$$

The space $Z$ is called a twisted sum of $Y$ and $X$. The exact sequence $\xymatrixcolsep{0.5cm}\xymatrix{0\ar[r]& Y  \ar[r]& Z \ar[r]& X \ar[r]&0}$ is said to \emph{split} if it is equivalent to the trivial sequence $\xymatrixcolsep{0.5cm}\xymatrix{0\ar[r]& Y  \ar[r]& Y\oplus_\infty X \ar[r]& X \ar[r]&0}$. Two exact sequences $\xymatrixcolsep{0.5cm}\xymatrix{0\ar[r]& Y  \ar[r]& Z \ar[r]& X \ar[r]&0}$
and $\xymatrixcolsep{0.5cm}\xymatrix{0\ar[r]& Y'  \ar[r]& Z' \ar[r]& X' \ar[r]&0}$ are called \emph{isomorphically equivalent} if and only if there exist isomorphisms $\alpha, \beta, \gamma$ forming  a commutative diagram
$$\xymatrix{
0\ar[r]&Y\ar[d]_\alpha\ar[r]& Z\ar[d]^\beta\ar[r]& X\ar[r]\ar[d]^\gamma & 0\\
0\ar[r]&Y'\ar[r]& Z'\ar[r]& X'\ar[r]& 0}$$
Two exact sequences  $\xymatrixcolsep{0.5cm}\xymatrix{0\ar[r]& Y  \ar[r]& Z \ar[r]& X \ar[r]&0}$
and $\xymatrixcolsep{0.5cm}\xymatrix{0\ar[r]& Y  \ar[r]& Z' \ar[r]& X \ar[r]&0}$ can easily be isomorphically equivalent but not equivalent (see \cite{hmbst}).

\subsection{Quasilinear maps and centralizers}

A homogeneous map $F: X\To Y$ (i.e., a map such that $F(\lambda x)=\lambda F(x)$ for every $x\in X$ and every scalar $\lambda$) is called $1$-quasilinear (resp. quasilinear) if for some $C>0$ and all finite sets $x_1, \dots, x_n \in X$ one has
$$\Big\|F \Big(\sum x_i\Big ) -\sum F(x_i)\Big\|\leq C \sum \|x_i\|$$
(resp. the condition above is satisfied for $n=2$). The infimum of the constants $C$ above is denoted $Q_1(F)$ (resp. $Q(F)$). Given a $1$-quasilinear map we can construct the space $Y\oplus_F X = ( Y\times X, \|\cdot\|_F)$ where $\|(y,x)\|_F=\|y - F x\| + \|x\|$ is a quasi-norm equivalent to a norm. All this remains true for quasilinear maps, except that  $\|(\cdot, \cdot)\|_F$ is now a quasi-norm no longer necessarily equivalent to a norm. Kalton's theorem \cite{kalt} (see also \cite{hmbst}) establishes that all quasilinear maps on Hilbert spaces are $1$-quasilinear and  $Q(\cdot)$ and  $Q_1(\cdot)$ are equivalent. Thus, when $F: X\To Y$ is a $1$-quasilinear map, $Y\oplus_F X$ is a Banach space and there is a natural exact sequence $$\xymatrix{0\ar[r]& Y \ar[r]^-{\imath} & Y\oplus_F X \ar[r]^-{\pi}& X \ar[r]&0}$$ in which $\imath(y)=(y,0)$ and $\pi(y,x)=x$. It turns out that every exact sequence $\xymatrixcolsep{0.5cm}\xymatrix{0\ar[r]& Y  \ar[r]& Z \ar[r]& X \ar[r]&0}$
of Banach spaces is equivalent to a sequence $\xymatrixcolsep{0.5cm}\xymatrix{0\ar[r]& Y  \ar[r]& Y\oplus_F X \ar[r]& X \ar[r]&0}$
for some 1-quasilinear map $F$ and, in particular, $Z$ is isomorphic to some $Y\oplus_F X$, something we will sometimes depict in the form
$$\xymatrix{
0\ar[r] & Y \ar[r]& Z \ar[r]& X\ar[r]\ar@{-->}@/_2pc/[ll]_F &0}$$

Indeed, given $\xymatrixcolsep{0.5cm}\xymatrix{0\ar[r]& Y  \ar[r]& Z \ar[r]^Q& X \ar[r]&0}$, a quasilinear map $F$ as above can be obtained as follows: pick a homogeneous bounded selector $b$ for the quotient map $Q$
(i.e. $Qb(x)=x$ for all $x\in X$), then pick a linear (not necessarily continuous) selector $\ell$ for $Q$ and set $F = b - \ell$. The map $F$ is obviously not unique. However, two $1$-quasilinear maps $F, F'$ generating equivalent exact sequences are such that $F - F' = B+L$ where $B: X\To Y$ is homogeneous and bounded (a homogeneous map $B:X\To Y$ is bounded when there is a constant $C$ such that $\|Bx\|\leq C\|x\|$ for all $x\in X$) and $L:X \To Y$ is linear. If we define two $1$-quasilinear maps $F, G: X\To Y$ as \emph{equivalent}, denoted $F \equiv G$, when $F - G = B+L$ as above, it can be shown that $F\equiv G$ if and only if the associated exact sequences are equivalent. We say that a quasilinear map $F$ is \emph{trivial} when $F\equiv 0$; equivalently, when the associated exact sequence splits. If $ \mbox{Lin} (X,Y)$ denotes the space of linear (not necessarily bounded) maps $X\to Y$, then $F\equiv 0$ if and only if $\mathrm{dist}(F, \mbox{Lin} (X,Y))<\infty$. We shall repeatedly use the pushout and pullback constructions. The reader can find a consistent study of those constructions in \cite{hmbst}. What we need to know here is that the pushout (resp. pullback) construction come defined by the commutative diagrams
$$\xymatrix{
0\ar[r]&Y\ar[d]_\alpha\ar[r]& Z\ar[d]\ar[r]& X\ar[r]\ar@{=}[d] \ar@{-->}@/_2pc/[ll]_F & 0\\
0\ar[r]&Y'\ar[r]& \PO\ar[r]& X\ar[r]& 0} \quad \quad\quad \xymatrix{
0\ar[r]&Y\ar[r]& Z\ar[r]& X\ar[r]\ar@{-->}@/_2pc/[ll]_F  & 0\\
0\ar[r]&Y \ar@{=}[u]\ar[r]& \PB \ar[u]\ar[r]& X'\ar[r]\ar[u]_\gamma& 0}$$

In quasilinear terms, the pushout construction corresponds to composition $\alpha F$ on the left and the pullback construction to composition $F\gamma$ on the right. Isomorphically equivalent sequences can also be described using quasilinear maps: $F$ and $G$ are isomorphic (written $F\sim G$) if and only if the associated exact sequences are isomorphically equivalent \cite{hmbst}; i.e., when there exists isomorphisms $\alpha, \gamma$ such that $\alpha F \equiv G\gamma$; equivalently, there exists a linear bounded operator $\beta$ making the following diagram commutative
$$\xymatrix{
0\ar[r] & Y \ar[d]_\alpha \ar[r]& Y\oplus_{F} X \ar[d]_\beta\ar[r]& X\ar[d]^\gamma\ar[r]\ar@{-->}@/_2pc/[ll]_F &0 \\
0\ar[r] &Y' \ar[r]& Y' \oplus_{G} X' \ar[r]& X' \ar@{-->}@/^2pc/[ll]^G  \ar[r]&0 }$$

Let $X$ be a Banach space contained in a vector space $\Sigma$. A homogeneous map $\Omega: X \To \Sigma$ defined on an $L_\infty$-module $X$ is an $L_\infty$-centralizer when there is a constant $C>0$ such that for every $\xi\in L_\infty$ and every $x\in X$ the difference
$\Omega(\xi x) - \xi \Omega (x)\in X$ and
$$\|\Omega(\xi x) - \xi \Omega (x)\| \leq C \|\xi\|_\infty \|x\|.$$

It is unknown in which contexts centralizers have to be quasilinear (which is the property required so that the twisted sum space is quasi-Banach), while quasilinear maps are not necessarily $1$-linear (which is the property required so that the twisted sum space is Banach). However, $L_\infty$-centralizers defined on K\"othe spaces are quasilinear  \cite[Lemma 3.12.3]{hmbst}, and these are the most important centralizers since they emerge from the deep theory created by Kalton \cite{kaltdiff,kaltmem} connecting centralizers with analytic  families of $L_\infty$-modules in the following way: if $\mathbb D$ denotes the unit complex circle (or the unit complex strip $\mathbb S = \{z\in \mathbb C: 0< \mathrm{Re} z <1\}$ when dealing with pairs of Banach spaces), given a suitable family $(X_\omega)_{\omega\in \partial \mathbb D}$ (a pair $(X_0, X_1)$) of complex Banach spaces then the complex interpolation method yields for every $z \in \mathbb D$ (for every $z\in \mathbb S$) a map $\Omega_{z_0}$, called the associated differential, defined on the interpolated space $X_{z_0}$. The connections between all these items is as follows: there is a  Banach space $\mathscr F$ of holomorphic functions $\mathbb D \To \Sigma$ (resp. holomorphic functions $\mathbb S \To \Sigma$), where $\Sigma$ is now a vector space that contains all the spaces $X_\omega$, in such a way that the interpolated space $X_z$ is obtained as the image of $\mathscr F$ by the evaluation map $\delta_{z_0}: \mathscr F \To \Sigma$. When the family is just a pair $(X_0, X_1)$ then it is customary to write $(X_0, X_1)_{z_0}$ for the interpolated space $X_{z_0}$. The differential $\Omega_{z_0}$ is obtained as
$$\Omega_{z_0} = \delta_{z_0}'B_{z_0}$$ where $B_{z_0}: X_{z_0} \to \mathscr F$ is a homogeneous bounded selector for $\delta_{z_0}$ and $\delta_{z_0}': \mathscr F\To \Sigma$ is the
evaluation at $z_0$ of the derivative $\delta_{z_0}(f)= f'(z_0)$. In this context, two different choices of selectors $B_{z_0}$ provide equivalent differentials. When the spaces of the family are $L_\infty$-modules then the differential $\Omega_{z_0}$ is a centralizer and, moreover, for every centralizer $\Omega$ on a (suitable) K\"othe function space $X$ there exists an analytic family $(X_\omega)_{\omega\in \partial \mathbb D}$ of K\"othe function spaces such that $X=X_{z_0}$ and $\Omega\equiv \Omega_{z_0}$.\medskip

Consider, for instance, the complex interpolation pair $(\ell_\infty, \ell_1)$ of $\ell_\infty$-modules. If $z_0= 1/2$ then it is well-known that $(\ell_\infty, \ell_1)_{1/2} = \ell_2$ and in this case $B(x)(z) = \mbox{sgn}(x) \left|x\right|^{2z}$ works as bounded homogeneous selector, so that the associated differential is the centralizer
$$\KP(x) = 2x \log \frac{|x|}{\|x\|_2}.$$

\subsection{Twisted Hilbert spaces}
 A twisted Hilbert space is a Banach space $Z$ admitting a subspace isomorphic to a Hilbert space in such a way that the corresponding quotient space is also a Hilbert space. Or, in homological terms, the middle space in an exact sequence $\xymatrix{0\ar[r] &H \ar[r]& Z\ar[r]& H' \ar[r] &0 }$ in which both $H$ and $H'$ are Hilbert spaces. A twisted Hilbert space is said to be nontrivial if it is not isomorphic to a Hilbert space, which, in homological terms, can be formulated as: the exact sequence $\xymatrixcolsep{0.5cm}\xymatrix{
0\ar[r] &H \ar[r]& Z\ar[r]& H' \ar[r] &0 }$ that defines it does not split. The first two examples and, to some extent, the most important ones of non-trivial twisted Hilbert spaces are:\medskip

\textbf{The \ELP space}. Let $(F_n)$ be a sequence of quasi-linear maps $F_n: \ell_2^n \To \ell_2^{m(n)}$ with
$\lim_{n\to \infty}  m(n)=\infty$; $Q(F_n)=1$ and $\lim_{n\to \infty}  \dist\Big (F_n, \mathfrak L(\ell_2^n, \ell_2^{m(n)})\Big )=\infty$. Then the nontrivial twisted Hilbert space generated by this sequence is

$$\ELP[F_n] = \ell_2\Big (\mathbb N, \ell_2^n\oplus_{F_n} \ell_2^{m(n)}\Big ).$$

These constructions are just an existence result: if a sequence $(F_n)$ like that exists,
a nontrivial twisted Hilbert space $\ELP[F_n]$ exists. And, by local results that can be seen in \cite{hmbst}, if a nontrivial twisted Hilbert space exists, a sequence $(F_n)$ exists. The critical point is how to obtain such a sequence $(F_n)$. Enflo, Lindenstrauss and Pisier constructed in \cite{elp} one specific such sequence by mere brute force, essentially as follows. For each $n$ let $\phi_n, P_n : \ell_2^{3^n} \rightarrow \ell_2^{3^n}$ be given by $P_n(x, y, z) = (x, y, 0)$, $\phi_1(x, y, z) = \Big(0, 0, \frac{x \left|y\right|}{(\left|x\right|^2 + \left|y\right|^2)^{\frac{1}{2}}}\Big)$ and, for $n\geq 2$,
$$\phi_{n+1}(x, y, z) = \Big(\phi_n(x), \phi_n(y), \frac{P_n(x) \|P_n(y)\|_2}{(\|P_n(x)\|^2 + \|P_n(y)\|^2)^{\frac{1}{2}}} \Big).$$

We set $\ELP = \ell_2\Big(\N, \ell_2^{3^n} \oplus_{\phi_n} \ell_2^{3^n}\Big)$
and observe:
\begin{itemize}\item $\ELP$ is a twisted Hilbert space in the obvious way

$$\xymatrix{0\ar[r] &\ell_2\Big(\N, \ell_2^{3^n}\Big)  \ar[r]& \ell_2\Big(\N, \ell_2^{3^n} \oplus_{\phi_n} \ell_2^{3^n}\Big)\ar[r]& \ell_2\Big(\N, \ell_2^{3^n}\Big) \ar[r] &0 }$$

\item It is nontrivial because it can be shown \cite{elp} (see also \cite{castgonz}) that $\ell_2^{3^n}$ is no less than
$\sqrt{n}$-complemented in $\ell_2^{3^n} \oplus_{\phi_n} \ell_2^{3^n}$.

\item A quasilinear map associated to this sequence is $\Omega_{\ELP}= \ell_2\Big(\N, \phi_n\Big)$, with the  meaning: given a finitely supported sequence $(x_n)$, $$\Omega_{\ELP}(x) = \Big(\phi_n(x_n) \Big).$$
\end{itemize}

\textbf{The Kalton-Peck space $Z_2$ and $Z(\varphi)$ spaces.} The second example of twisted Hilbert space that is not a Hilbert space is $\ell_2\oplus_{\KP} \ell_2$. This is the celebrated Kalton-Peck $Z_2$ space \cite{kaltpeck} (see \cite{hmbst} to discover its remarkable properties) generated by the centralizer $\KP$, which will be called the Kalton-Peck map. The property we need to mention here is that the exact sequence
$\xymatrixcolsep{0.5cm}\xymatrix{0\ar[r]&\ell_2 \ar[r]& Z_2 \ar[r]& \ell_2 \ar[r]&0}$ is \emph{singular} and \emph{cosingular}, in the sense that the quotient map is a strictly singular operator and the embedding is a strictly cosingular operator. An exact sequence defined by a quasilinear map $\Omega$ is singular if and only if no restriction of $\Omega$ to an infinite dimensional subspace is trivial \cite{hmbst}. The $\KP$ centralizer belongs to the family of centralizers introduced in \cite{kaltpeck}: using a suitable Lipschitz function $\varphi : \R \rightarrow \R$ or $\varphi : \R \rightarrow \mathbb{C}$ define
$$\Omega_\varphi(x) = x \varphi\Big(- \log \frac{|x|}{\|x\|_2}\Big).$$
The twisted Hilbert space generated by this centralizer will be denoted $Z(\varphi)$. Observe that $Z_2$ is the space $Z(\varphi)$ obtained from the choice $\varphi(t)=t$ for $t\geq 0$.\medskip

Some specific choices of the Lipschitz function $\varphi$ will be of importance for this paper
\begin{itemize}
\item For a given real number $\alpha$, the choice  $\varphi_\alpha(t)= t^{1+i\alpha }$ produces the spaces $Z_2(\alpha)$ obtained in \cite{kcomplex}, for which Kalton shows that when $\alpha\neq \beta$ then $Z_2(\alpha)$ is real-isomorphic but not complex-isomorphic to $Z_2(\beta)$.
\item A function $\varphi$ will be called bi-Lipschitz if there exist constants $c,C>0$ such that $c|t-s|\leq |\varphi(t) - \varphi(s)|\leq C |t - s|$. It was shown in \cite{cd} that there is a coneable set $\mathcal C$ (i.e., it contains a convex cone generated by an infinite linearly independent set) of bi-Lipschitz functions satisfying additional conditions, which were used to obtain a family $\{Z(\varphi): \varphi\in \mathcal C\}$ of twisted Hilbert spaces having the following properties:
\begin{enumerate}
\item Each $Z(\varphi)$ is isomorphic to its dual since it is generated by a centralizer.
\item The exact sequences $\xymatrixcolsep{0.5cm}\xymatrix{0\ar[r]&\ell_2 \ar[r]& Z(\varphi) \ar[r]& \ell_2 \ar[r]&0}$ are singular and cosingular \cite[Proposition 3.3]{cd}.
\item Given $\varphi, \psi\in \mathcal C$, if $\varphi$ is not a multiple of $\psi$ then $Z(\varphi)$ is not a subspace (or a quotient) of $Z(\psi)$ \cite[Theorem 4.1]{cd}.
\end{enumerate}
\end{itemize}

Let us improve (3). Consider the following class: an operator $T: \ell_2\To X$ will be called \emph{almost-compact} if $Te_n\To 0$. The reason behind the name is explained now:

\begin{lemma} An operator $T: \ell_2\To X$ is almost-compact if and only if for every infinite $N\subset \N$ there exists an infinite subset $M\subset N$ such that the restriction $T|_{\ell_2(M)}$ is compact.
\end{lemma}
\begin{proof} Suppose that $T$ is almost compact. We assume without loss of generality that $N=\N$. Since $T$ is almost compact there is an infinite subset $M\subset N$ such that $(\|Te_n\|)_{n\in M} \in \ell_2$; now, let $N(\varepsilon) \in \N$ be chosen such that $\left( \sum_{N(\varepsilon)}^\infty \|Te_n\|^2 \right)^{1/2}<\varepsilon$. Let $\imath: \ell_2(M) \To \ell_2(\N)$ be the canonical inclusion. To show that  $T|_{\ell_2(M)} = T\imath$ is compact we show that $(T|_{\ell_2(M)})^* = \imath^*T^*: X^* \To \ell_2(M)$ is compact. If $x^*\in B_{X^*}$ then
$$\|\imath^* T^* x^*\| = \Big( \sum_{m\in M} |\langle T^* x^*, e_m\rangle|^2\Big)^{1/2} = \Big( \sum_{m\in M} |\langle x^*, Te_m\rangle|^2\Big)^{1/2}
\leq \Big( \sum^{N(\varepsilon)} \|T e_m\|^2\Big)^{1/2}  + \varepsilon$$
which shows that $\imath^*T$ is approximable by the sequence $(\pi_NT)_N$ of finite rank operators, where $\pi_N: \ell_2(\N)\To \ell_2(M)$ is the projection onto the first $N$ elements of $M$.

The converse implication is immediate: since compact operators transform weakly null sequences into norm null sequences the condition in the Lemma
implies that every subsequence of $(\|Te_n\|)_n$ contains a subsequence convergent to $0$, hence the whole sequence is convergent to $0$.\end{proof}

The existence of almost-compact operators $\ell_2\To \ell_2$ that are not compact shows that almost-compact operators do not form an operator ideal. The following example has been provided by Manuel Gonz\'alez:\medskip

\noindent\textbf{Example.} Pick the orthonormal sequence $(y_n)$ given by $y_1 = e_1$ and $y_n = 2^{-n/2} e_{2^n} + \cdots + e_{2^{n+1} -1}$ for $n > 1$. Define $T : \ell_2 \To \ell_2$ in the form
$$Te_k = \left\{
           \begin{array}{ll}
             e_1, & \hbox{$k=1$;} \\
             2^{-n/2} e_{2^n}, & \hbox{$2^n \leq k< 2^{n+1}$.}
           \end{array}
         \right.$$
$T$ is a bounded operator that obviously satisfies $Te_n\To 0$ but, at the same time, $T$ is an  isomorphism on the closed subspace generated by the $y_n$'s since $Ty_n = e_n$.

\begin{proposition}\label{ss} Given $\varphi, \psi\in \mathcal C$, if $\varphi$ is not a multiple of $\psi$ then every operator $\tau: Z(\varphi) \To Z(\psi)$ has almost-compact restriction to $\ell_2$. Consequently, $Z(\varphi)$ is not a subspace of $Z(\psi)$.\end{proposition}
\begin{proof} From \cite[Proposition 4.3]{cd} (see the first line of the proof) we get that for every subsequence $(e_{n_k})$ of the canonical basis of $\ell_2$ we have $\inf_n \|\tau(e_{n_k}, 0)\| = 0$; therefore $\|Te_n\|\To 0$. This means that the restriction $\tau|_{\ell_2}$ must be almost-compact. The second assertion follows from that.
\end{proof}


Since the $Z(\varphi)$ spaces are isomorphic to their square by construction, $Z(\varphi)$ is not either a subspace or a quotient of $Z(\psi)\oplus \ell_2$ since this is a subspace of $Z(\psi)\oplus Z(\psi)$, in turn isomorphic to $(\phi)$. Moreover, $Z(\varphi)$ does not contain complemented copies of $\ell_2$ \cite[Lemma 1]{complex} and therefore $Z(\varphi)$ is not isomorphic to $Z(\varphi)\oplus \ell_2$. More specific properties of this family of spaces will be mentioned when required.

\section{Two problems of Cabello}

\subsection{Problem 1} Twisted Hilbert spaces generated from complex interpolation scales, i.e., twisted Hilbert spaces whose associated quasilinear map is a differential, are isomorphic to their duals \cite{ccc}. A delicate second reading shows that twisted Hilbert spaces generated by $L_\infty$-\emph{centralizers} are isomorphic to their duals \cite{cabeann}. All twisted Hilbert spaces known so far, with the exception of the \ELP space \cite{elp}, are obtained in this form. This motivates Cabello to formulate in \cite[p. 796]{cabeann} -- reproduced in \cite[Problem 1, Section 4.5]{free} -- the question: \emph{Is every twisted Hilbert space isomorphic to its dual?} The answer is no:

\begin{theorem} There exist twisted Hilbert spaces not isomorphic to their duals.\end{theorem}

The proof requires some preparation. Pick two nontrivial twisted Hilbert spaces $Z(\Phi)$ and $Z(\Psi)$ generated by the quasilinear maps $\Phi, \Psi$ as we will depict in the next diagram. In the following pullback diagram we omit the initial and final $0's$

$$\xymatrix{
\ell_2 \ar[r]& Z(\Phi) \ar[r]^{\pi_{\Phi}}& \ell_2\ar@{-->}@/_2pc/[ll]_\Phi\ar@{-->}@/^2pc/[dd]^\Psi \\
\ell_2 \ar@{=}[u] \ar[r]& Z(\Phi, \Psi) \ar[u]^u\ar[r]_v& Z(\Psi) \ar[u]_{\pi_{\Psi}}\\
& \ell_2 \ar[u]\ar@{=}[r]& \ell_2 \ar[u]}$$

The diagram conceals a ``diagonal" exact sequence (see \cite{hmbst}; also \cite{castmoreLN})
$$\xymatrix{
0\ar[r] &\ell_2\oplus\ell_2  \ar[r]& Z(\Phi, \Psi) \ar[r]^{\eta}& \ell_2 \ar@{-->}@/_2pc/[ll]_{\Phi\oplus\Psi}\ar[r]&0}$$
where $(\Phi\oplus\Psi)(x) = (\Phi x, \Psi x)$ and $\eta=\pi_{\Phi} u = \pi_{\Psi} v$. Therefore, if $\Delta: \ell_2 \To \ell_2 \oplus \ell_2$ is the diagonal map $\Delta(x) = (x, x)$, the previous diagonal sequence coincides with the lower pullback sequence in the diagram
$$\xymatrix{
0\ar[r] &\ell_2\oplus \ell_2 \ar[r]& A\oplus B  \ar[r]^{\pi_{\Phi}, \pi_{\Psi}}& \ell_2\oplus \ell_2\ar[r]\ar@{-->}@/_2pc/[ll]_{(\Phi, \Psi)}&0 \\
0\ar[r] &\ell_2\oplus \ell_2 \ar@{=}[u] \ar[r]& Z(\Phi, \Psi) \ar[u]\ar[r]& \ell_2 \ar[u]_\Delta\ar[r]&0\\
}$$
since $\Phi\oplus \Psi = (\Phi, \Psi)\Delta$. The dual diagram

$$\xymatrix{
\ell_2^* \ar[d]_{\pi_{\Psi}^*}\ar[r]^{\pi_{\Phi}^*}& A^*\ar[d] \ar[r]& \ell_2^*\ar[d]\\
B^*\ar[d]\ar[r]& Z(\Phi, \Psi)^* \ar[d]\ar[r]& \ell_2^*\\
\ell_2 \ar@{=}[r]& \ell_2 &}$$
shows that $Z(\Phi, \Psi)^*$ contains both $A^*$ and $B^*$.\medskip

It is time to specialize. Pick $\Phi, \Psi$ centralizers, so that $Z(\Phi)^* \simeq Z(\Phi)$ and $Z(\Psi)^*\simeq Z(\Psi)$. The proof consists in showing that the right choice of $\Phi, \Psi$ yields a space $Z(\Phi, \Psi)$ that does not contain either $Z(\Phi)$ or $Z(\Psi)$. Therefore $Z(\Phi, \Psi)^* \simeq Z(\Phi, \Psi)$ is impossible.\medskip

With that purpose in mind, let $\mathcal C$ be the coneable family of Lipschitz embeddings obtained in \cite{cd}, and for each $\varphi\in \mathcal C$ let $Z(\varphi)$ be the associated twisted Hilbert space. Pick two $\varphi, \psi\in \mathcal C$ such that $\varphi$ is not a multiple of $\psi$, and let $\Omega_{\varphi}, \Omega_{\psi}$ denote the associated centralizers generating $Z(\varphi)$ and $Z(\psi)$. Form the pullback diagram

$$\xymatrix{
\ell_2 \ar[r]& Z(\varphi) \ar[r]^{\pi_{\varphi}}& \ell_2 \\
\ell_2 \ar@{=}[u] \ar[r]& Z(\varphi, \psi) \ar[u]^{u_\varphi}\ar[r]^{u_{\psi}}& Z(\psi) \ar[u]_{\pi_\psi}\\
& \ell_2 \ar[u]\ar@{=}[r]& \ell_2 \ar[u]}$$
and consider the standard representation of the pullback space
$$Z(\varphi, \psi) = \{(x, y, z, w) \in Z(\varphi) \oplus Z(\psi) : y = w\} = \{(x, y, z, y) \in Z(\varphi) \oplus Z(\psi)\}$$
endowed with the quasinorm
$$
\|(x, y, z, y)\|_{Z(\varphi, \psi)} = \|(x, y)\|_{\Omega_{\varphi}} + \|(z, w)\|_{\Omega_{\psi}} = \|x - \Omega_{\varphi} y \|_2 + \|z - \Omega_{\psi} y\|_2 + 2 \|y\|_2
$$
In this representation $\eta(x, y, z, y) = y$; and since $\eta = \pi_\varphi u_\varphi = \pi_{\psi} v_\psi$, the operator $\eta$ is strictly singular. Notice that we have two bounded operators $P : Z(\varphi, \psi) \rightarrow Z(\varphi)$ and $Q : Z(\varphi, \psi) \rightarrow Z(\psi)$, which are the restrictions of the natural projections of $Z(\varphi) \oplus Z(\psi)$ onto $Z(\varphi)$ and $Z(\psi)$, respectively.

\begin{lemma}\label{lema}
There is no bounded operator $S : Z(\varphi) \rightarrow Z(\varphi, \psi)$ such that $\|S(e_n, 0)\|_{Z(\varphi, \psi)} > c > 0$ for every $n$ and some $c > 0$.
\end{lemma}
\begin{proof}
First, notice that by a standard perturbation argument, taking a subsequence of $(e_n)_n$, we may suppose that $S(e_n, 0) = (a_n, b_n, c_n, b_n)$ and $S(0, e_n) = (d_n, f_n, g_n, f_n)$ for every $n$, where $(a_n), (b_n), (c_n), (d_n), (f_n), (g_n)$ are block sequences of $(e_n)$ with the property that
\[
\mbox{supp } (a_n) , \mbox{supp } (b_n), \mbox{supp } (c_n), \mbox{supp } (d_n), \mbox{supp } (f_n), \mbox{supp } (g_n) \subset I_n
\]
and $I_1 < I_2 < I_3 < ...$ are intervals of natural numbers.\medskip

Since $\eta$ is strictly singular, $\eta S$ is strictly singular and therefore its restriction to (any copy of) $\ell_2$ is compact. Consequently
$\eta S (e_n, 0) = b_n \to 0$. Thus, we may assume that $S(e_n, 0) = (a_n, 0, c_n, 0)$. Since $\|S(e_n, 0)\|_{Z(\varphi, \psi)} = \|a_n\|_2 + \|c_n\|_2 > c$ for every $n$, given $k \geq 2$ and $n$, we have $\|a_n\|_2 > \frac{c}{k}$ or $\|c_n\|_2 > \frac{c}{k}$. Therefore, passing to a subsequence, we have the alternative:\medskip

\textbf{Case 1:} $\|c_n\|_2 > d > 0$ for every $n$. In that case the operator $Q S : Z(\varphi) \rightarrow Z(\psi)$ satisfies $\|Q  S(e_n, 0)\|_{Z(\psi)} > d > 0$ for every $n$ in open contradiction with \cite[Proposition 4.3]{cd} (see the first line of the proof) .\medskip

\textbf{Case 2:} $\|c_n\|_2 \rightarrow 0$. Then, again by a perturbation, we may suppose that $S(e_n, 0) = (a_n, 0, 0, 0)$.

\medskip

\emph{Claim:} $\|f_n\|_2 \rightarrow 0$. Otherwise, we can assume without loss of generality that $\|f_n\|_2 > \delta > 0$ for every $n$
and the operator $T (y) = \sum_n y_n f_n$ is an isomorphism onto its image on $\ell_2$. Since
$$S(x, y) = \Big(\sum_n (x_n a_n + y_n d_n), \sum_n y_n f_n, \sum_n y_n g_n, \sum_n y_n f_n\Big)$$
we obtain
$$ \left \|\sum_n y_n g_n - \Omega_{\psi} \Big(\sum_n y_n f_n\Big)\right\|_2 \leq \|S\| \|y\|_2.$$

Thus,$\Omega_{\psi}$ would be trivial on $T[\ell_2]$, which is impossible because it is singular. That proves the claim.\\

Therefore, we may suppose that $S(0, e_n) = (d_n, 0, g_n, 0)$, so that
$$S(x, y) = (\sum_n (x_n a_n + y_n d_n), 0, \sum_n y_n g_n, 0).$$
Given $N$ let $y^N = \sum\limits_{n=1}^N \frac{1}{\sqrt{N}} e_n$. Apply $S$ to $(\Omega_{\varphi}(y^N), y^N)$ to obtain
$$\left\|\sum_{n=1}^N \frac{1}{\sqrt{N}} \varphi(\log \sqrt{N}) a_n + \sum_{n=1}^N \sqrt{\frac{1}{N}} d_n\right\|_2 \leq \|S\|
$$
Notice that $\|d_n\|_2 \leq \|S(0, e_n)\| \leq \|S\|$. Since $(d_n)$ is a block sequence, we get
$$\left\|\sum_{n=1}^N \sqrt{\frac{1}{N}} d_n\right\|_2 \leq \|S\|$$
Since $\|a_n\|_2 = \|S(e_n, 0)\| > c$, $(a_n)$ is a block sequence and $\varphi$ is bi-Lipschitz we get
$$\left\|\sum_{n=1}^N \frac{1}{\sqrt{N}} \varphi(\log \sqrt{N}) a_n\right\|_2 \geq C \log N$$
for some $C$ and every $N$. Thus
\begin{eqnarray*}
C \log N & \leq & \left\|\sum_{n=1}^N \frac{1}{\sqrt{N}} \varphi(\log \sqrt{N}) a_n\right\|_2 \\
        & \leq & \left\|\sum_{n=1}^N \frac{1}{\sqrt{N}} \varphi(\log \sqrt{N}) a_n + \sum_{n=1}^N \sqrt{\frac{1}{N}} d_n\right\|_2 + \left\|\sum_{n=1}^N \sqrt{\frac{1}{N}} d_n\right\|_2 \\
        & \leq & 2 \|S\|
\end{eqnarray*}
for every $N$, which is impossible.\end{proof}

\subsection{The family $Z(\varphi, \psi)$} We have just shown that given $\varphi, \psi\in \mathcal C$, one not multiple of the other, we can construct a twisted Hilbert space $Z(\varphi, \psi)$ not isomorphic to its dual. We show now:

\begin{theorem}\label{zfg} Let $\varphi_1, \psi_1, \varphi_2, \psi_2 \in \mathcal C$ be. If either no nonzero linear combination of $\varphi_1, \psi_1$ and $\varphi_2$ is bounded or no nonzero linear combination of $\varphi_1, \psi_1$ and $\psi_2$ is bounded then $Z(\varphi_1, \psi_1)$ is not a subspace of $Z(\varphi_2, \psi_2)$.
\end{theorem}
\begin{proof} A combination of the diagonal technique \cite[Proposition 1.c.8]{lindtzaf} as in \cite{kcomplex}, see also, \cite{cd}, and the fact that the spaces $Z(\varphi, \psi)$ have a 3-UFDD given by $E_n = [(e_n, 0, 0, 0), (0, 0, e_n, 0), (0, e_n, 0, e_n)]$ shows that if there is an operator (isomorphic embedding) $Z(\varphi_1, \psi_1) \rightarrow Z(\varphi_2, \psi_2)$ then there is an operator (isomorphic embedding)  $T : Z(\varphi_1, \psi_1) \rightarrow Z(\varphi_2, \psi_2)$ defined by an scalar matrix
$$T = \begin{pmatrix}
\alpha_1 & \alpha_2 & \alpha_3 \\
\alpha_4 & \alpha_5 & \alpha_6 \\
0 & 0 & \alpha_7
\end{pmatrix} : Z(\varphi_1, \psi_1) \rightarrow Z(\varphi_2, \psi_2)
$$
in the form $T(x, y, z, y) = (\alpha_1 x + \alpha_2 z + \alpha_3 y, \alpha_7 y, \alpha_4 x + \alpha_5 z + \alpha_6 y, \alpha_7 y).$\\

\noindent \textbf{Claim} \emph{If no nonzero linear combination of $\varphi_1, \psi_1$ and $\varphi_2$ is bounded or no nonzero linear combination of $\varphi_1, \psi_1$ and $\psi_2$ is bounded then $\alpha_1 = \alpha_2 = \alpha_4 = \alpha_5 = \alpha_7 = 0$.}\medskip

\begin{proof}[Proof of the Claim] Pick the normalized element  $s_N = \frac{1}{\sqrt{N}}\sum_{n=1}^N e_n$, and notice that that $\Omega_\varphi(s_N) = \varphi(\log \sqrt{N}) s_N$. Since
$$T(\Omega_{\varphi_1} x, x, \Omega_{\psi_1} x, x) = (\alpha_1 \Omega_{\varphi_1} x + \alpha_2 \Omega_{\psi_1} x + \alpha_3 x, \alpha_7 x, \alpha_4 \Omega_{\varphi_1} x + \alpha_5 \Omega_{\psi_1} x + \alpha_6 x, \alpha_7 x)$$
we obtain that $\left|\alpha_1 \varphi_1(\log \sqrt{N}) + \alpha_2 \psi_1(\log \sqrt{N}) + \alpha_3 - \alpha_7 \varphi_2 (\log \sqrt{N})\right|$ is bounded above by
\begin{eqnarray*}
\|(\alpha_1 \Omega_{\varphi_1} + \alpha_2 \Omega_{\psi_1} + \alpha_3 - \alpha_7 \Omega_{\varphi_2}) s_N\|
&\leq&\Big\|T\Big(\Omega_{\varphi_1} s_N, s_N, \Omega_{\psi_1} s_N, s_N\Big)\Big\|\\
&\leq& 2\|T\|.\end{eqnarray*}
Since $\varphi_1, \psi_1$ and $\varphi_2$ are bi-Lipschitz, this estimate gives us
$$\sup_{t > 0} \Big|\alpha_1 \varphi_1(t) + \alpha_2 \psi_1(t) - \alpha_7\varphi_2 (t) \Big| < \infty$$
so that $\alpha_1 = \alpha_2 = \alpha_7 = 0$. Similarly, we obtain $\alpha_4 = \alpha_5 = 0$.\end{proof}

To conclude the proof, let $T : Z(\varphi_1, \psi_1) \rightarrow Z(\varphi_2, \psi_2)$ be an isomorphic embedding, that we can assume of the form $$\begin{pmatrix}
\alpha_1 & \alpha_2 & \alpha_3 \\
\alpha_4 & \alpha_5 & \alpha_6 \\
0 & 0 & \alpha_7
\end{pmatrix}$$
The claim shows that the matrix is actually
$$\begin{pmatrix}
0 & 0& \alpha_3 \\
0 & 0 & \alpha_6 \\
0 & 0 & 0
\end{pmatrix}$$
which implies that $T[ Z(\varphi_1, \psi_1)] \subset \ell_2\oplus \ell_2$ and $Z(\varphi_1, \psi_1)$ would be a Hilbert space.\end{proof}

\subsection{Problem 2.}\label{sec:problem2} The second question of Cabello, in part connected with Problem 1 but having an intrinsic interest in itself, is: \emph{Is every quasi-linear map between Hilbert spaces isomorphic to a centralizer?} Here the key of the question is the word \emph{isomorphic}. Indeed, there is nothing special in a quasilinear map not equivalent to a centralizer: if $\Omega$ is a centralizer and we decompose $\ell_2 = \ell_2 \oplus \ell_2$ and set $\Tilde{\Omega} = (0, \Omega)$, then $\Tilde{\Omega}$ cannot be equivalent to a centralizer. And it is so because equivalence just means to be a centralizer in a prescribed basis. Thus, in general
\begin{proposition}
Let $F : \ell_2 \rightarrow \ell_2$ be a nontrivial quasilinear map with the following property: there are closed subspaces $U, V \subset \ell_2$ such that $\ell_2 = U \oplus V$, $F(U) \subset V$ and $F|_V = 0$. Then $F$ cannot be equivalent to a centralizer.
\end{proposition}
\begin{proof} In \cite{cfg} it is proved that every centralizer $\Phi$ is boundedly equivalent to a contractive centralizer $\Phi'$, i.e., for every $x\in \ell_2$, the support of $\Phi' x$ is contained in the support of $x$. Now, if the quasilinear map $F$ is equivalent to a centralizer, we may find a linear map $L : \ell_2 \rightarrow \ell_2$ and a contractive centralizer $\Omega$ such that $F - \Omega - L$ is bounded. Write $L|_U = (L_1, L_2)$. Since $F(U) \subset V$ and $\Omega(U) \subset U$, we actually obtain that $F|_U - L_2$ is bounded. Since $F|_V = 0$, that implies that $F$ is trivial.\end{proof}

In particular, $\Omega_{\ELP}$ is not equivalent to a centralizer. We do not know however if it is isomorphic to a centralizer. We notice, however, that the negative solution of Problem 1 automatically gives a negative solution of Problem 2. Indeed, it is well known that every differential $\Omega$ on $\ell_2$ satisfies $\Omega^* \equiv - \Omega$ (see \cite{ccc} for a generalization).
Thus, if $\Omega$ is isomorphically equivalent to a centralizer, then $\ell_2 \oplus_{\Omega} \ell_2$ is isomorphic to its dual.

The next proposition shows that it is possible to obtain a negative solution to Problem 2 even though the twisted Hilbert space obtained is isomorphic to its dual. Observe that
$$\Omega \oplus \Omega = (\Omega, \Omega)\Delta = \Delta \Omega.$$


\begin{proposition}\label{prop:2} Given a singular centralizer $\Omega$ on $\ell_2$, the quasilinear map $\Delta\Omega$ is not isomorphic to a centralizer.\end{proposition}
\begin{proof} Let us first observe that the quasiliear map $\Delta \Omega =(\Omega, \Omega)\Delta$ generating the sequence
$$\xymatrix{
0\ar[r] &\ell_2\oplus \ell_2\ar[r]& Z(\Omega, \Omega) \ar[r]& \ell_2\ar[r]&0}$$
is singular by construction. However, the dual sequence
$$\xymatrix{
0\ar[r] &\ell_2^*\ar[r]& Z(\Omega, \Omega)^* \ar[r]& \ell_2^*\oplus\ell_2^* \ar[r]&0}$$
which is generated by $(\Delta \Omega)^* = \Omega^* \Delta^*= - \Omega \Sigma$, where $\Sigma = \Delta^*$ is the map $\Sigma(x,y)= x+y$, cannot be singular
because its restriction to  $\{ (x, - x): x\in \ell_2^*\}$ is obviously trivial.\medskip

Now, if $\Delta \Omega$ were isomorphic to a centralizer $\Phi$, namely $\Phi = \alpha \Delta \Omega \gamma$ for two given isomorphisms $\alpha, \gamma$, then
$$\gamma^* \Omega^*\Delta^* \alpha ^* = \Phi^* = -\Phi = - \alpha \Delta \Omega \gamma \Longrightarrow \Omega^* \Delta^* = {\gamma^*}^{-1} \alpha \Delta \Omega \gamma {\alpha^*}^{-1} $$
and thus $\Delta \Omega$ would be isomorphic to its dual $\Omega^*\Delta^*$; and since two isomorphic sequences must be simultaneously singular
$\Omega^*\Delta^*$ would be singular, which it is not.\end{proof}

This same argument could be repeated with a pair $(\Phi, \Psi)$ of centralizers such that:
\begin{itemize}
\item One of them is singular
\item They are projectively equivalent on some infinite dimensional subspace of $\ell_2$.
\end{itemize}

Indeed, the first condition, and the fact that both $\Phi, \Psi$ are pushout of $(\Phi, \Psi)\Delta$, imply that $(\Phi, \Psi)\Delta$ must be singular. The dual of $(\Phi, \Psi)\Delta$ is $\sigma(\Phi, \Psi)^* = \Psi^* + \Phi^* = - (\Psi + \Phi)$. This map cannot be singular by the second condition.


\section{Incomparable elements in $\Ext(\ell_2, \ell_2)$}

Consider the vector space $\Ext(\ell_2, \ell_2)$ of twisted Hilbert sequences (modulo equivalence) endowed with the following order: $\Phi\leq \Psi$ if $\psi$ is a pushout of $\Phi$, that is, there is an operator $\tau\in \mathfrak L(\ell_2)$ such that $\tau \Phi \equiv \Psi$. In \cite{ccfmcomplex} it was shown that the is no minimum for this order --i.e., there is no twisted Hilbert sequence $\Phi$ such that any other twisted Hilbert sequence has the form $\tau\Phi$ for some operator $\tau$ -- which is remarkable since other vector spaces such as $\Ext(c_0, C[0,1])$ have minimum. Subsequently, it was shown in \cite{ext2} the existence of two incomparable centralizers $\Phi, \Psi$ for that order. Let us improve those results showing, among other things, that all exact sequences of the coneable family of $Z(\varphi)$ spaces are incomparable.

\begin{proposition} Given $\varphi, \psi\in \mathcal C$, no commutative diagrams
\begin{equation}\label{po}\xymatrix{
0\ar[r] & \ell_2 \ar[d] \ar[r]& Z(\varphi) \ar[d] \ar[r]& \ell_2\ar@{=}[d]\ar[r]&0 \\
0\ar[r] &\ell_2 \ar[r]&  Z(\psi) \ar[r]& \ell_2 \ar[r]&0 }\end{equation}
or
\begin{equation}\label{pb}\xymatrix{
0\ar[r] & \ell_2 \ar@{=}[d] \ar[r]& Z(\varphi) \ar[r]& \ell_2\ar[r]&0 \\
0\ar[r] &\ell_2 \ar[r]&  Z(\psi) \ar[u]\ar[r]& \ell_2 \ar[r]\ar[u]&0 }\end{equation}
are possible.\end{proposition}

\begin{proof} Proposition \ref{ss} shows that any operator $T: Z(\varphi)\To \Z(\psi)$ must have almost compact restriction to $\ell_2$. Let $M\subset \N$ be such that $\tau = T|_{\ell_2(M)}$ is compact. Now, \cite[Proposition 4.5]{compact} shows that the lower sequence in a commutative diagram
$$\xymatrix{
0\ar[r] & Y \ar[d]_\tau \ar[r]& Z \ar[d] \ar[r]& X\ar@{=}[d]\ar[r]&0 \\
0\ar[r] &Y' \ar[r]&  Z' \ar[r]& X \ar[r]&0 }$$
in which $\tau$ is compact cannot have strictly singular quotient map. The maps $\Omega_\phi$ have the property that $\supp \Omega_\phi(x) \subset \supp x$, hence the restriction of $\Omega_\varphi$ to $\ell_2(M)$ generates a sequence that we will denote
$\xymatrixcolsep{0,5cm}\xymatrix{0\ar[r] &\ell_2(M) \ar[r]&  Z(\psi)(M) \ar[r]& \ell_2(M) \ar[r]&0 }$ and has strictly singular quotient map.
A diagram like (\ref{po}) thus generates a diagram
$$\xymatrix{
0\ar[r] & \ell_2(M) \ar[d]_\tau \ar[r]& Z(\varphi)(M) \ar[d] \ar[r]& \ell_2(M)\ar@{=}[d]\ar[r]&0 \\
0\ar[r] &\ell_2(M) \ar[r]&  Z(\psi)(M) \ar[r]& \ell_2(M) \ar[r]&0}$$
with $\tau$ compact, and that is impossible. The second assertion follows by duality.\end{proof}

The same argumentation and the proof of Theorem \ref{zfg} show that if either no nonzero linear combination of $\varphi_1, \psi_1$ and $\varphi_2$ is bounded or no nonzero linear combination of $\varphi_1, \psi_1$ and $\psi_2$ is bounded then no commutative diagram
$$\xymatrix{
0\ar[r] & \ell_2 \oplus \ell_2\ar[d] \ar[r]& Z(\varphi_1, \psi_1) \ar[d] \ar[r]& \ell_2\ar@{=}[d]\ar[r]&0 \\
0\ar[r] &\ell_2\oplus \ell_2 \ar[r]&  Z(\varphi_2, \psi_2) \ar[r]& \ell_2 \ar[r]&0 }$$
is possible.\smallskip

On the other hand the construction of the sequence $\Phi\oplus\Psi$ is a particular case of the construction of products in the category $\mathfrak Q^{\ell_2}$ of exact sequences whose quotient space is $\ell_2$ as described in \cite[4.3]{castmoreLN}. Therefore, $\Phi\oplus \Psi \leq \Phi$ and $\Phi\oplus \Psi \leq \Psi$, showing that the nonexistence of initial element is an infinity question. Infinite products also exist in $\mathfrak Q^{\ell_2}$, and thus the product of a sequence of elements $F_\gamma: \ell_2\To \ell_2$, $\gamma\in \Gamma$ exists and is the quasilinear map $\ell_\infty(\Gamma,  F_\gamma): \ell_2 \To \ell_\infty (\Gamma, \ell_2)$: indeed, if $\pi_\theta$ denotes the projection onto the $\theta^{th}$ component, then $\pi_\theta \ell_\infty(\Gamma,  F_\gamma) = F_\theta$. The lack of initial element shows that the range cannot be reduced, in general, to be a Hilbert space. However, when the family is countable we can obtain a kind of product in the isomorphic sense: pick some $\sigma\in \ell_2$ with $\sigma_n\neq 0$ for every $n$ and observe that the product $\ell_\infty(\N, \sigma_n F_n)$ of $(\sigma_n F_n)$ yields a twisted Hilbert sequence such that each $F_k$ is isomorphic to $\pi_k \ell_\infty(\N, \sigma_n F_n)$.

\section{Further considerations and unanswered questions}

\subsection{Twisted Hilbert sequences not isomorphically equivalent to its dual}

The identity $\Phi^* \equiv - \Phi$ yields that $\ell_2\oplus_\Phi \ell_2$ is isomorphic to its dual and the information that the two sequences below are isomorphically equivalent:
$$\xymatrix{
0\ar[r] & \ell_2 \ar[d]_{-1} \ar[r]& Z(\Phi) \ar[d] \ar[r]& \ell_2\ar@{=}[d]\ar[r]\ar@{-->}@/_2pc/[ll]_\Phi &0 \\
0\ar[r] &\ell_2 \ar[r]&  Z(\Phi)^* \ar[r]& \ell_2 \ar@{-->}@/^2pc/[ll]^{\Phi^*}  \ar[r]&0 }$$

This suggests a weaker form of Problem 1: \emph{Is every twisted Hilbert sequence isomorphically equivalent to its dual?} The answer is of course no: consider the examples $(\Phi\oplus\Psi)\Delta$ answering Problem 1 and the examples $\Delta \Omega$ answering Problem 2 (during the proof it was shown that $\Delta\Omega$ cannot be isomorphic to $\Omega^*\Sigma$). However, even if the sequence generated by $\Delta\Omega$ cannot be isomorphic to that generated by $\Omega^*\Sigma$, the twisted Hilbert space $Z(\Omega, \Omega)$ is isomorphic to its dual because we have a commutative diagram

$$\xymatrix{
\ell_2 \ar[r]\ar[d]_\Delta& Z(\Omega) \ar[d]\ar[r]& \ell_2 \ar@{=}[d]\ar@{-->}@/_2pc/[ll]_{\Omega}  \ar[r]&0 \\
\ell_2\oplus \ell_2 \ar@{-->}[ur]\ar[d]\ar[r]& Z(\Omega, \Omega)\ar[d]\ar[r]& \ell_2 \ar[r]&0\\
\ell_2 \ar@{=}[r]& \ell_2&}$$
in which the dotted arrow exists because the left vertical sequence splits. Thus the inclusion $\ell_2\to Z(\Omega)$ can be extended to $\ell_2\oplus \ell_2$ and therefore the middle vertical sequence splits, which yields $Z(\Omega, \Omega) \simeq  Z(\Omega) \oplus \ell_2$. Since $Z(\Omega)\simeq Z(\Omega)^*$, also $Z(\Omega, \Omega) \simeq Z(\Omega, \Omega)^*$.\medskip

A rather general criterium can be given. Following \cite{ideal}, given a twisted Hilbert sequence
$$\xymatrix{0 \ar[r] & \ell_2 \ar[r]^\imath & Z \ar[r]^Q & \ell_2 \ar[r] & 0}$$
an operator $f \in \OP L(\ell_2)$ (resp. $g \in \OP L(\ell_2)$) will be called \emph{$Z$-liftable} (resp. \emph{$Z$-extensible}) if there is an operator $F : \ell_2 \rightarrow Z$ (resp. $G : Z \rightarrow \ell_2$) making the corresponding diagram
$$\xymatrix{Z \ar[r]^Q& \ell_2\\
    &\ell_2 \ar[ul]^F\ar[u]_f } \quad \quad \quad \quad\mathrm{or} \quad  \quad\quad\quad  \xymatrix{\ell_2 \ar[r]^\imath\ar[d]_g &Z\ar[dl]^G\\
\ell_2}$$
commute. The space of $Z$-liftable operators will be denoted $\mathcal L^Z$ and that of $Z$-extensible operators, $\mathcal E_Z$. We have

\begin{proposition} Consider a twisted Hilbert sequence isomorphically equivalent to its dual in the form
$$\xymatrix{
0\ar[r]&Y\ar[d]_\alpha\ar[r]^\imath& Z\ar[d]^\beta\ar[r]^Q& X\ar[r]\ar[d]^\gamma & 0\\
0\ar[r]&Y^*\ar[r]_{Q^*}& Z^*\ar[r]_{\imath^*}& X^*\ar[r]& 0}$$
Then\begin{enumerate}
\item $ f \in \mathcal L^Z \Longleftrightarrow f^* \in \mathcal E_{Z^*} \Longleftrightarrow \gamma f \in \mathcal L^{Z^*} \Longleftrightarrow f^* \gamma^* \in \mathcal E_Z.$
\item $g \in \mathcal E_Z \Longleftrightarrow g^* \in \mathcal L^{Z^*}\Longleftrightarrow g \alpha^{-1} \in \mathcal E_{Z^*} \Longleftrightarrow (\alpha^{-1})^* g^* \in \mathcal L^Z.$
\end{enumerate}
\end{proposition}
\begin{proof} Let $f \in \mathcal L^Z$ and let $F$ be its lifting. The first equivalence in (1) is obvious, as it is the rest after a careful contemplation of the diagram
$$\xymatrix{
&&& \ell_2\ar[d]^f\ar[dl]_F&\\
0\ar[r]&Y\ar[d]_\alpha\ar[r]^\imath& Z\ar[d]^\beta\ar[r]^Q& X\ar[r]\ar[d]^\gamma & 0\\
0\ar[r]&Y^*\ar[r]_{Q^*}& Z^*\ar[r]_{\imath^*}& X^*\ar[r]& 0}$$
and some sleight with arrows: $Q \beta^{-1} = \gamma^{-1} \imath^*$ implies $\gamma^{-1} \imath^* \beta F = f$, so that $\gamma f \in \mathcal L^{Z^*}$. Exchanging the roles of $Z$ and $Z^*$ we get the rest. For the equivalences in (2), just stare at the diagram
$$\xymatrix{
& \ell_2\ar[d]_g\ar[dr]^G&\\
0\ar[r]&Y\ar[d]_\alpha\ar[r]^\imath& Z\ar[d]^\beta\ar[r]^Q& X\ar[r]\ar[d]^\gamma & 0\\
0\ar[r]&Y^*\ar[r]_{Q^*}& Z^*\ar[r]_{\imath^*}& X^*\ar[r]& 0}$$
and, again, $\beta \imath = Q^* \alpha$ implies $G \beta^{-1} Q^* \alpha = g$, so that $g \alpha^{-1}\in \mathcal E_{Z^*}$. Exchanging the roles of $Z$ and $Z^*$ we get the rest.\end{proof}

It is worth reformulating the result in the negative:
\begin{cor}\label{corollary}
Consider a twisted Hilbert sequence $\xymatrixcolsep{0.5cm}\xymatrix{0 \ar[r] & \ell_2 \ar[r] & Z \ar[r] & \ell_2 \ar[r] & 0}$. If either \begin{itemize}
\item For every isomorphism $\gamma$ on $\ell_2$ there is $f \in \mathcal L^Z$ such that $f^* \gamma \not\in \mathcal E_Z$; or
\item For every isomorphism $\alpha$ on $\ell_2$ there is $g \in \mathcal  E_Z$ such that $\alpha g^* \not\in \mathcal L^Z$
\end{itemize}
then the sequence is not isomorphically equivalent to its dual.\end{cor}

\subsection{About being equivalent to its dual} Corollary \ref{corollary} can be reformulated for equivalent mappings:

\begin{cor}\label{corollary2} Consider a twisted Hilbert sequence $\xymatrixcolsep{0.5cm}\xymatrix{0 \ar[r] & \ell_2 \ar[r] & Z \ar[r] & \ell_2 \ar[r] & 0}$. If either\begin{itemize}
    \item There is $f \in \mathcal L^Z$ such that $f^* \not\in \mathcal E_Z$; or
    \item There is $g \in \mathcal E_Z$ such that $g^* \not\in \mathcal L^Z$
\end{itemize}
then the sequence is not equivalent to its dual.
\end{cor}


A question floating in the air is what occurs with the generic examples $\ell_2(\mathbb N, F_n)$, which are never singular (even when they are generated by a centralizer). In particular, the original $\ELP$ construction still contains some mysteries inside. We already know that $\Omega_{\ELP} $ is not equivalent to a centralizer. Let us show now that, as expected,

\begin{lemma}  $\Omega_{\ELP}$ is not  equivalent to its dual.
\end{lemma}
\begin{proof} A direct adaptation of \cite[Lemma 2.3]{ideal} yields

\begin{lemma} Let $\xymatrixcolsep{0.5cm}\xymatrix{0\ar[r]& \ell_2  \ar[r]& Z \ar[r]& \ell_2 \ar[r]&0}$ be a twisted Hilbert sequence defined by the quasilinear map $\Omega$.
\begin{enumerate}
    \item[(a)] $f \in \mathcal L^Z$ if and only if $\mathrm{dist}(\Omega  f, \emph{\texttt{Lin}} (\ell_2)) < \infty$.
    \item[(b)] $f \in \mathcal E_Z$ if and only if $\mathrm{ dist} (f \Omega, \emph{\texttt{Lin}}(\ell_2)) < \infty$.
\end{enumerate}
\end{lemma}

Now recall the construction of $\Omega_{\ELP} = \ell_2(\mathbb N, \phi_n)$. It is enough to consider a family of normalized maps $f_n \in \OP L(\ell_2^{3^n})$ such that
$$\sup d(\phi_n f_n, \texttt{Lin}(\ell_2^{3^n})) < \infty \quad \quad  \mathrm{and}\quad \quad\sup d(f_n^* \phi_n, \texttt{Lin} (\ell_2^{3^n})) = \infty$$

To that end, set $f_1(x, y, z) = (0, 0, z)$ and $f_{n+1}(x, y, z) = (f_n(x), f_n(y), z)$. Then each $f_n$ is self-adjoint, $\phi_n f_n = 0$ while $f_n \phi_n = \phi_n$ and, therefore, $\ELP$ is not equivalent to its dual.\end{proof}

However, the major question left open in this paper is:
\begin{prob*}
    Is the Enflo-Lindenstrauss-Pisier space $\ELP$ isomorphic to its dual? Is it isomorphic to a centralizer?
\end{prob*}

\end{document}